\documentclass{amsart}
\usepackage{fixltx2e}
\usepackage[utf8]{inputenc}
\usepackage[T1]{fontenc}
\usepackage{amssymb,amsthm,amsmath}
\usepackage[english]{babel}
\usepackage[all]{xy}
\usepackage{calrsfs}
\usepackage{hyperref}
\usepackage{mathbbol}
\usepackage{mathabx}
\usepackage{xcolor}

\theoremstyle{plain}
\newtheorem{theorem}[subsection]{Theorem}
\newtheorem{lemma}[subsection]{Lemma}
\newtheorem{proposition}[subsection]{Proposition}
\newtheorem{corollary}[subsection]{Corollary}

\theoremstyle{definition}
\newtheorem{definition}[subsection]{Definition}

\theoremstyle{remark}
\newtheorem{remark}[subsection]{Remark}

\newenvironment{tfae}
{
\begin{enumerate}}
{\end{enumerate}}

\newcommand{\comp}{\circ}
\newcommand{\defn}{\textbf}
\newcommand{\join}{\vee}
\newcommand{\tensor}{\otimes}

\DeclareMathOperator{\Kernel}{Ker}
\newcommand{\Ker}{\Kernel}

\newcommand{\C}{\ensuremath{\mathcal{C}}}

\newcommand{\K}{\ensuremath{\mathbb{K}}}

\newcommand{\coc}{\mathit{coc}}

\newcommand{\BiAlg}{\ensuremath{\mathsf{BiAlg}}}
\newcommand{\BiAlgcoc}{\ensuremath{\mathsf{BiAlg}_{\K, \coc}}}
\newcommand{\CoAlg}{\ensuremath{\mathsf{CoAlg}}}
\newcommand{\Mon}{\ensuremath{\mathsf{Mon}}}
\newcommand{\Gp}{\ensuremath{\mathsf{Gp}}}
\newcommand{\Hopf}{\ensuremath{\mathsf{Hopf}}}
\newcommand{\Set}{\ensuremath{\mathsf{Set}}}

\newdir{>>}{{}*!/3.5pt/:(1,-.2)@^{>}*!/3.5pt/:(1,+.2)@_{>}*!/7pt/:(1,-.2)@^{>}*!/7pt/:(1,+.2)@_{>}}
\newdir{ >>}{{}*!/8pt/@{|}*!/3.5pt/:(1,-.2)@^{>}*!/3.5pt/:(1,+.2)@_{>}}
\newdir{ |>}{{}*!/-3.5pt/@{|}*!/-8pt/:(1,-.2)@^{>}*!/-8pt/:(1,+.2)@_{>}}
\newdir{ >}{{}*!/-8pt/@{>}}
\newdir{>}{{}*:(1,-.2)@^{>}*:(1,+.2)@_{>}}
\newdir{<}{{}*:(1,+.2)@^{<}*:(1,-.2)@_{<}}

\def\pullback{% with thanks to Valerian Even
 \ar@{-}[]+R+<6pt,-1pt>;[]+RD+<6pt,-6pt>%
 \ar@{-}[]+D+<1pt,-6pt>;[]+RD+<6pt,-6pt>}

\def\ophalfsplitpullback{%
 \ar@{-}[]+R+<6pt,-1pt>;[]+RD+<6pt,-6pt>%
 \ar@{-}[]+D+<.5ex,-6pt>;[]+RD+<6pt,-6pt>}

\begin{document}

\title{A note on split extensions of bialgebras}
\author{Xabier García-Martínez}
\address[Xabier García-Martínez]{Department of Algebra, University of Santiago de Compostela, 15782 Santiago de Compostela, Spain.}
\email{xabier.garcia@usc.es}

\author{Tim Van~der Linden}
\address[Tim Van~der Linden]{Institut de
Recherche en Math\'ematique et Physique, Universit\'e catholique
de Louvain, che\-min du cyclotron~2 bte~L7.01.02, B--1348
Louvain-la-Neuve, Belgium}
\email{tim.vanderlinden@uclouvain.be}

\thanks{This work was partially supported by Ministerio de Economía y Competitividad (Spain), grant MTM2016-79661-P (AEI/FEDER, UE, support included). The first author was also supported by Xunta de Galicia, grant GRC2013-045 (European FEDER support included), by an FPU scholarship of the Ministerio de Educación, Cultura y Deporte (Spain) and by a Fundaci\'on Barri\'e scholarship. The second author is a Research Associate of the Fonds de la Recherche Scientifique--FNRS}

\begin{abstract}
We prove a universal characterization of Hopf algebras among cocommutative bialgebras over an algebraically closed field: a cocommutative bialgebra is a Hopf algebra precisely when every split extension over it admits a join decomposition. We also explain why this result cannot be extended to a non-cocommutative setting.
\end{abstract}

%\keywords{Cocommutative bialgebra, Hopf algebra, protomodular category, split extension}


%\date{\today}

\maketitle

\section{Introduction}
An elementary result in the theory of modules says that in any short exact sequence 
\[
\xymatrix{0 \ar[r] & K \ar[r]^-{k} & X \ar@<.5ex>[r]^-{f} & Y \ar@{->}@<.5ex>[l]^-{s} \ar[r]& 0}
\qquad\qquad
f\comp s=1_{Y}
\]
where the cokernel $f$ admits a section $s$, the middle object $X$ decomposes as a direct sum $X\cong K\oplus Y$. If, however, the given sequence is a short exact sequence of, say, groups or Lie algebras, then this is of course no longer true: then we can at most deduce that $X$ is a semidirect product $K\rtimes Y$ of $K$ and $Y$. In a fundamental way, this interpretation depends on, or even amounts to, the fact that $X$ is generated by its subobjects $k(K)$ and $s(Y)$. One may argue that, in a non-additive setting, the join decomposition $X=k(K)\join s(Y)$ in the lattice of subobjects of~$X$ is what replaces the direct sum decomposition, valid for split extensions of modules. 

When the given split extension is a sequence of cocommutative bialgebras (over a commutative ring with unit~$\K$), we may ask ourselves the question whether such a join decomposition of the middle object in the sequence always exists. Although kernels are not as nice as one could expect \cite{Agore,AndDev}, it is not difficult to see that \emph{if $Y$ is a Hopf algebra} then the answer is yes. 

The main point of this note is that this happens \emph{only then}, at least when~$\K$ is an algebraically closed field. We shall prove, in other words, the following new universal characterization of cocommutative Hopf algebras among cocommutative bialgebras over $\K$: 
\begin{quote}
\emph{All split extensions over a bialgebra $Y$ admit a join decomposition if and only if $Y$ is a Hopf algebra.}
\end{quote}
This result is along the lines of, and is actually a variation on, a similar characterization of groups among monoids, recently obtained in~\cite{MRVdL-TCOGAM,GM-ACS}. There the authors show that all split extensions (of monoids) over a monoid $Y$ admit a join decomposition if and only if $Y$ is a group. 

In fact something stronger than the existence of a join decomposition may be proved in a more general context; this will be the subject of Section~\ref{Internal monoids}, where we explore some basic aspects of split extensions of cocommutative bialgebras. In particular, we show that over a Hopf algebra, all split extensions of cocommutative bialgebras admit a join decomposition (Corollary~\ref{Corollary Join Decomposition}). In Section~\ref{Characterization} we focus on the other implication and prove that among cocommutative bialgebras over an algebraically closed field, only Hopf algebras admit join decompositions of their split extensions (Theorem~\ref{Theorem 2}). In the final Section~\ref{Cocommutativity} we explain why the constraint that the bialgebras in this characterization are cocommutative is essential. As it turns out, in a non-cocommutative setting, even the very weakest universal join decomposition condition is too strong.

\section{Split extensions over Hopf algebras}\label{Internal monoids}
A \defn{split extension} in a pointed category with finite limits $\C$ is a diagram
\[
\xymatrix{K \ar[r]^-{k} & X \ar@<.5ex>[r]^-{f} & Y \ar@{->}@<.5ex>[l]^-{s}}
\]
where $k$ is a kernel and $s$ is a section of $f$. So $f\comp s=1_{Y}$, but a priori we are not asking that $f$ is a cokernel of $k$, so that $(k,f)$ is a short exact sequence, and this is not automatically the case. We do always have that $K$ and $Y$, considered as subobjects of $X$, have a trivial intersection. Indeed, using that $k$ is the pullback of~$0\to Y$ along $f$, it is easy to check that the pullback of~$k$ and~$s$ is zero.

In this general context, a join of two subobjects may not always exist, but the concept introduced in the next definition expresses what we want, and agrees with the condition that $X=k(K)\join s(Y)$ whenever that expression makes sense---as it does in any regular category with binary coproducts, for instance~\cite{Borceux-Bourn}.

\begin{definition}
A pair of arrows $(k, s)$ with the same codomain $X$ is \defn{jointly extremally epimorphic} when the arrows $k$ and $s$ cannot both factor through one and the same proper subobject of $X$: whenever we have a diagram
\[
\xymatrix@!0@=4em{ & M \ar@{ >->}[d]^- m \\
 K \ar[r]_-k \ar[ur] & X & Y \ar[l]^-s \ar[ul]}
\]
where $m$ is a monomorphism, necessarily $m$ is an isomorphism. We say that a split extension as above is \defn{strong} when $(k,s)$ is a jointly extremally epimorphic pair; the couple $(f,s)$ is then called a \defn{strong point}. When we say that a split extension \defn{admits a join decomposition}, we mean that it is strong. 

The given split extension is said to be \defn{stably strong} (the couple $(f,s)$ is a \defn{stably strong point}) when all of its pullbacks (along any morphism $g\colon W\to Y$) are strong. Following \cite{MRVdL-TCOGAM}, we say that an object $Y$ is \defn{protomodular} when all split extensions over $Y$ are stably strong.
\end{definition}

\begin{remark}
It is easily seen~\cite{MRVdL-TCOGAM} that the split epimorphism $f$ in a strong point $(f,s)$ is always the cokernel of its kernel $k$. This means, in particular, that all split extensions over a protomodular object $Y$, as well as all of their pullbacks, are (split) short exact sequences which admit a join decomposition.
\end{remark}

\begin{remark}
When all objects in $\C$ are protomodular, $\C$ is a \defn{protomodular category} in the sense of~\cite{Bourn1991}. Next to Barr exactness, this is one of the key ingredients in the definition of a semi-abelian category~\cite{Janelidze-Marki-Tholen}, and crucial for results such as the \emph{$3\times 3$ Lemma}, the \emph{Snake Lemma}, the \emph{Short Five Lemma}~\cite{Bourn2001,Borceux-Bourn}, or the existence of a Quillen model category structure for homotopy of simplicial objects~\cite{VdLinden:Simp}. Typical examples are the categories of groups, Lie algebras, crossed modules, loops, associative algebras, etc. As recently shown in~\cite{GKV,Kadjo-PhD}, also the category of cocommutative Hopf algebras over a field of characteristic zero is semi-abelian.
\end{remark}

Given a category with finite products $\C$, we write $\Mon(\C)$ for the category of internal monoids, and $\Gp(\C)$ for the category of internal groups in $\C$. For a commutative ring with unit $\K$, we let $\CoAlg_{\K,\coc}$ denote the category of cocommutative coalgebras over $\K$. It is well known~\cite{Sweedler} that there is an equivalence between the category $\BiAlgcoc$ of cocommutative bialgebras over $\K$ and $\Mon(\CoAlg_{\K,\coc})$, which restricts to an equivalence between the category $\Hopf_{\K,\coc}$ of cocommutative Hopf algebras over $\K$ and $\Gp(\CoAlg_{\K,\coc})$. This is easily seen using that in $\CoAlg_{\K,\coc}$ the product $X\times Y$ is $X\tensor Y$ and $1$ is $\K$. 

\begin{theorem}\label{Theorem 1}
Let $\C$ be a category with finite limits. If $Y \in \Gp(\C)$ then all split extensions in $\Mon(\C)$ over $Y$ are stably strong. In other words, any internal group in $\C$ is a protomodular object in $\Mon(\C)$.
\end{theorem}
\begin{proof}
Consider in $\Mon(\C)$ the commutative diagram
	\[
	\xymatrix@C=4em{
		& \Ker (\pi_1) \ar[d]_-{l} \ar@/_1.2pc/[dl] \ar@{=}[r] & \Ker(f) \ar[d]^{k} \\
		M \ar@{{ >}->}[r]^-m & W \times_{Y} X \ophalfsplitpullback \ar@<.5ex>[d]^(.6){\pi_1} \ar[r]^-{\pi_2} & X \ar@<.5ex>[d]^-f \\
		& W \ar@<.5ex>[u]^-{\langle 1_{W}, s\circ g\rangle} \ar[r]_-g \ar@/^1.2pc/[lu] & Y \ar@<.5ex>[u]^-s
	}
	\]
where the bottom right square is a pullback, $m$ is a monomorphism, and $Y$ is an internal group. We shall see that $m$ is an isomorphism. Since only limits are considered, the whole commutative diagram is sent into a category of presheaves of sets by the Yoneda embedding, in such a way that the internal groups and internal monoids in it are mapped to ordinary groups and monoids, respectively. Since the Yoneda embedding reflects isomorphisms, it now suffices to give a proof in~$\Set$. There, it is easy to see that $m$ is an isomorphism, since every element $(w, x)$ of~$W \times_{Y} X$ can be written as $(1, x\cdot s(g(w)^{-1})) \cdot (w, sg(w))$, where clearly the first element belongs to the kernel of $\pi_1$ and the second one comes from $W$.
\end{proof}

\begin{corollary}\label{Corollary Join Decomposition}
Cocommutative Hopf algebras are protomodular in $\BiAlgcoc$.\hfill\qed
\end{corollary}
%\begin{proof}
%It suffices to use the equivalences of categories $\BiAlgcoc\simeq \Mon(\CoAlg_{\K,\coc})$ and $\Hopf_{\K,\coc}\simeq \Gp(\CoAlg_{\K,\coc})$.
%\end{proof}

It follows that, over a Hopf algebra, split extensions of bialgebras are well-behaved; not only are they short exact sequences, but it is also not hard to see that the \emph{Split Short Five Lemma} holds for them, so that equivalences classes of split extensions may be considered as in ordinary group cohomology.

\section{A universal characterization of cocommutative Hopf algebras}\label{Characterization}

The converse is less straightforward. In the case of groups and monoids ($\C=\Set$ in Theorem~\ref{Theorem 1}), it was shown in \cite{MRVdL-TCOGAM} (resp.\ in \cite{GM-ACS}) that all points in $\Mon$ over~$Y$ are stably strong (resp.\ strong) if and only if $Y$ is a group.
However, those proofs involve coproducts, and so a Yoneda embedding argument as in Theorem~\ref{Theorem 1} would not work. 

We now let $\K$ be an algebraically closed field.
We consider the adjoint pair 
\[
\xymatrix{\BiAlgcoc \ar@<1ex>[r]^-{G} \ar@{}[r]|-{\top} & \Mon \ar@<1ex>[l]^-{\K[-]}}
\]
where the left adjoint $\K[-]$ is the monoid algebra functor and the right adjoint $G$ sends a bialgebra $B$ (with comultiplication $\Delta_{B}$ and counit $\varepsilon_{B}$) to its monoid of grouplike elements $G(B)=\{x\in B\mid \text{$\Delta_{B}(x)=x\tensor x$ and $\varepsilon_{B}(x)=1$}\}$. 

\begin{lemma}\label{Lemma K[-] Preserves Monos}
$\K[-]$ preserves monomorphisms.
\end{lemma}
\begin{proof}
The functor $\K[-]$ sends any monoid monomorphism to a bialgebra morphism of which the underlying vector space map is an injection. 
\end{proof}

Our aim is to prove that $G$ preserves protomodular objects: then for any protomodular bialgebra $B$, the monoid of grouplike elements $G(B)$ is a group, so that $B$ is a Hopf algebra by~\cite[8.0.1.c and 9.2.5]{Sweedler}.

\begin{proposition}\label{Proposition Unit Split Monic}
For any monoid $M$ we have $G(\K[M])\cong M$. For any bialgebra~$B$, the counit $\epsilon_{B}\colon \K[G(B)]\to B$ of the adjunction at $B$ is a split monomorphism with retraction $\pi_{B}\colon B\to \K[G(B)]$, determined in a way which is functorial in $B$.
\end{proposition}
\begin{proof}
The first statement follows immediately from the definition of $\K[M]$, while the second depends on \cite[8.0.1.c and 8.1.2]{Sweedler}.
\end{proof}

Since protomodular objects are closed under retracts~\cite{MRVdL-TCOGAM}, it follows that if~$B$ is a protomodular bialgebra, then so is $\K[G(B)]$.

\begin{proposition}\label{Proposition G Preserves JSE Pairs}
The functor $G$ preserves jointly extremally epimorphic pairs.
\end{proposition}
\begin{proof}
Let $(k,s)$ be a jointly extremally epimorphic pair in $\BiAlgcoc$. Then the commutativity of the diagram
\[
\xymatrix{\K[G(K)] \ar[r]^-{\K[G(k)]} & \K[G(X)] & \K[G(Y)] \ar[l]_-{\K[G(s)]}\\
K \ar[r]_-{k} \ar[u]^-{\pi_{K}} & X \ar[u]^-{\pi_{X}} & Y \ar[u]_-{\pi_{Y}} \ar[l]^-{s}}
\]
obtained via Proposition~\ref{Proposition Unit Split Monic} and the fact that the upward pointing arrows are split epimorphisms imply that the pair $(\K[G(k)],\K[G(s)])$ is jointly extremally epimorphic. Now suppose that $m$ is a monomorphism making the diagram on the left
\[
\vcenter{\xymatrix{&M \ar@{{ >}->}[d]^-{m}\\
G(K) \ar[ru] \ar[r]_-{G(k)} & G(X) & G(Y) \ar[lu] \ar[l]^-{G(s)}}}
\qquad
\vcenter{\xymatrix{&\K[M] \ar@{{ >}->}[d]^-{\K[m]}\\
\K[G(K)] \ar[ru] \ar[r]_-{\K[G(k)]} & \K[G(X)] & \K[G(Y)] \ar[lu] \ar[l]^-{\K[G(s)]}}}
\]
commute. Applying $\K[-]$ we obtain the diagram on the right, in which $\K[m]$ is a monomorphism by Lemma~\ref{Lemma K[-] Preserves Monos}. Since, by the above, the bottom pair is jointly extremally epimorphic, we see that $\K[m]$ is an isomorphism. But then also $m=G(\K[m])$ is an isomorphism, which proves our claim that $(G(k),G(s))$ is a jointly extremally epimorphic pair. 
\end{proof}

\begin{proposition}\label{Proposition G Preserves Protomodular Objects}
If all split extensions over a bialgebra $Y$ are strong, then all split extensions over $G(Y)$ are strong. In particular, $G$ preserves protomodular objects. 
\end{proposition}
\begin{proof}
Consider a split extension
\[
\xymatrix{K \ar[r]^{k} & X \ar@<.5ex>[r]^-{f} & G(Y) \ar@<.5ex>[l]^-{s}}
\]
over $G(Y)$. We apply the functor $\K[-]$, then take the kernel of $\K[f]$ to obtain the split extension of bialgebras
\[
\xymatrix{L \ar[r]^-{l} & \K[X] \ar@<.5ex>[r]^-{\K[f]} & \K[G(Y)]. \ar@<.5ex>[l]^-{\K[s]}}
\]
From Proposition~\ref{Proposition Unit Split Monic} it follows that all split extensions over $\K[G(Y)]$ are strong. Hence $(l,\K[s])$ is a jointly extremally epimorphic pair. Applying the functor $G$, we regain the original split extension, since $G$ is a right adjoint, thus preserves kernels; but $G$ also preserves jointly extremally epimorphic pairs by Proposition~\ref{Proposition G Preserves JSE Pairs}, so that the pair $(k,s)$ is jointly extremally epimorphic. As a consequence, all split extensions over the monoid $G(Y)$ are strong, and $G(Y)$ is protomodular~\cite{GM-ACS}. 
\end{proof}

\begin{theorem}\label{Theorem 2}
If $\K$ is an algebraically closed field and $Y$ is a cocommutative bialgebra over $\K$, then the following conditions are equivalent:
\begin{tfae}
\item $Y$ is a Hopf algebra;
\item in $\BiAlgcoc$, all split extensions over $Y$ admit a join decomposition;
\item $Y$ is a protomodular object in $\BiAlgcoc$.
\end{tfae} 
\end{theorem}
\begin{proof}
(i) implies (iii) is Theorem~\ref{Theorem 1}, and (ii) is obviously weaker than (iii). For the proof that (ii) implies (i), suppose that all split extensions over $Y$ admit a join decomposition. Then Proposition~\ref{Proposition G Preserves Protomodular Objects} implies that in $\Mon$ all split extensions over $G(Y)$ are strong. Hence $G(Y)$ is a group by the result in~\cite{GM-ACS}, which makes $Y$ a Hopf algebra by~\cite[8.0.1.c and 9.2.5]{Sweedler}. 
\end{proof}

\begin{remark}
This implies that the category $\BiAlgcoc$ cannot be protomodular: otherwise all bialgebras would be Hopf algebras. In particular, the \emph{Split Short Five Lemma} is not generally valid for bialgebras.
\end{remark}

\section{On cocommutativity}\label{Cocommutativity}
In this final section we study what happens beyond the cocommutative setting. Here $\K$ is a field.

All objects in the category of cocommutative $\K$-bialgebras satisfy a certain weak join decomposition property: being a category of internal monoids (in~$\CoAlg_{\K,\coc}$), the category $\BiAlg_{\K,\coc}$ is \defn{unital} in the sense of~\cite{Borceux-Bourn}. Given an object $Y$, it is said to be a \defn{unital object}~\cite{MRVdL-TCOGAM} when every split extension of the type
\[
\xymatrix{X \ar@<-.5ex>[r]_-{\langle 1_{X},0\rangle} & X\times Y \ar@<-.5ex>[l]_-{\pi_{X}} \ar@<.5ex>[r]^-{\pi_{Y}} & Y \ar@{->}@<.5ex>[l]^-{\langle 0,1_{Y}\rangle}}
\]
is strong. Notice how this condition is symmetric in $X$ and $Y$. So protomodular objects are always unital of course, but in fact this condition is weak enough to be satisfied by all cocommutative bialgebras over $\K$.

Let us now leave the cocommutative setting and ask ourselves what it means for an object~$Y$ in~$\BiAlg_{\K}$ to be unital---a very weak thing to ask, compared with the condition that all split extensions over $Y$ are (stably) strong. 

\begin{proposition}
If $Y$ is a unital object of $\BiAlg_{\K}$, then for every object $X$ we have an isomorphism $X\times Y\cong X\tensor Y$.
\end{proposition}
\begin{proof}
Given any bialgebra $X$ we may consider the diagram
\[
\xymatrix@=3em{X & \ar[l]_-{\rho_{X}}^-{\cong} X\tensor \K \ar@<-.5ex>[r]_-{1_{X}\tensor \eta_{Y}} & X\tensor Y \ar@<-.5ex>[l]_-{1_{X}\tensor \varepsilon_{Y}} \ar@<.5ex>[r]^-{\varepsilon_{X}\tensor 1_{Y}} & \K\tensor Y \ar[r]^-{\lambda_{Y}}_-{\cong} \ar@<.5ex>[l]^-{\eta_{X}\tensor 1_{Y}} & Y. }
\]
We are first going to prove that the comparison morphism
\[
m=\langle \rho_{X}\circ(1_{X}\tensor \varepsilon_{Y}),\lambda_{Y}\circ(\varepsilon_{X}\tensor 1_{Y})\rangle \colon X\tensor Y\to X\times Y
\]
is a monomorphism. 

Note that it is almost never an injection; for instance, taking $X=Y$ to be a tensor algebra $T(V)$ (with counit $\varepsilon_{T(V)}(v)=0$ for $v\in V$) yields easy counterexamples. However, in the category $\BiAlg_{\K}$, monomorphisms need not be injective \cite{NuTo,Agore2}. 

Let $h\colon {Z\to X\tensor Y}$ be a morphism of bialgebras. We write 
\[
f=\rho_{X}\circ (1_{X}\tensor \varepsilon_{Y})\circ h\colon Z\to X
\qquad
\text{and}
\qquad
g=\lambda_{Y}\circ (\varepsilon_{X}\tensor 1_{Y})\circ h\colon Z\to Y.
\]
It suffices to prove that $h=(f\tensor g)\circ \Delta_{Z}$ \emph{as vector space maps} for our claim to hold. Indeed, if $h$ and $h'$ induce the same $f$ and $g$, then the given equality of vector space maps proves that $h=h'$. 

Since $h$ is a coalgebra map, we have that $\Delta_{X\tensor Y}\circ h = (h\tensor h)\circ \Delta_{Z}$. Writing $\tau_{X,Y}\colon{X\tensor Y\to Y\tensor X}$ for the twist map, we calculate:
\begin{align*}
&(f\tensor g)\circ\Delta_{Z} \\
&= (\rho_{X}\tensor \lambda_{Y})\circ (1_{X}\tensor \varepsilon_{Y}\tensor \varepsilon_{X}\tensor 1_{Y}) \circ (h\tensor h)\circ\Delta_{Z}\\
&= (\rho_{X}\tensor \lambda_{Y})\circ (1_{X}\tensor \varepsilon_{Y}\tensor \varepsilon_{X}\tensor 1_{Y}) \circ \Delta_{X\tensor Y}\circ h\\
&= (\rho_{X}\tensor \lambda_{Y})\circ (1_{X}\tensor \varepsilon_{Y}\tensor \varepsilon_{X}\tensor 1_{Y}) \circ (1_{X}\tensor \tau_{X,Y}\tensor 1_{Y})\circ(\Delta_{X}\tensor \Delta_{Y})\circ h\\
&= (\rho_{X}\tensor \lambda_{Y})\circ (1_{X}\tensor \varepsilon_{X}\tensor \varepsilon_{Y}\tensor 1_{Y}) \circ(\Delta_{X}\tensor \Delta_{Y})\circ h\\
&= (\rho_{X}\tensor \lambda_{Y})\circ (\rho_{X}^{-1}\tensor \lambda_{Y}^{-1})\circ h=h.
\end{align*}

It follows that $m$ is a monomorphism. Moreover, $m$ makes the diagram
\[
\xymatrix@!0@R=4em@C=6em{ & X\tensor Y \ar@{ >->}[d]^-m \\
 X \ar[r]_-{\langle 1_{X},0\rangle} \ar[ur]^-{(1_{X}\tensor \eta_{Y})\circ \rho_{X}^{-1}} & X\times Y & Y \ar[l]^-{\langle 0,1_{Y}\rangle} \ar[ul]_-{(\eta_{X}\tensor 1_{Y})\circ \lambda_{Y}^{-1}}}
\]
commute. The assumption that $Y$ is unital tells us that $m$ is an isomorphism. 
\end{proof}

This immediately implies that any unital object $Y$ in $\BiAlg_{\K}$ has to be cocommutative, since $\Delta_{Y}\colon Y\to Y\tensor Y$ is the morphism of bialgebras $\langle 1_{Y},1_{Y}\rangle\colon Y\to Y\times Y$. In particular, the category $\BiAlg_{\K}$ is not unital, so it cannot be protomodular, and not even Mal'tsev~\cite{Borceux-Bourn}. 

However, the situation is actually much worse, since it almost never happens that $X\tensor Y$ is the product of $X$ in $Y$ in the category of all $\K$-bialgebras---not even when both $X$ and $Y$ are cocommutative. In fact, $\K$ itself cannot be a protomodular object in $\BiAlg_{\K}$, since this would imply that all objects of $\BiAlg_{\K}$ are unital~\cite{MRVdL-TCOGAM}. As we have just seen, this is manifestly false. 

The same holds for the category $\Hopf_{\K}$ of Hopf algebras over $\K$. At first this may seem to contradict results in~\cite{Molnar} on split extensions of Hopf algebras. We must keep in mind, though, that for a Hopf algebra $H$, the map $\langle 1_{H},0\rangle$ in the  diagram
\[
\xymatrix{H \ar@<-.5ex>[r]_-{\langle 1_{H},0\rangle} & H\times H \ar@<-.5ex>[l]_-{\pi_{1}} \ar@<.5ex>[r]^-{\pi_{2}} & H \ar@{->}@<.5ex>[l]^-{\langle 0,1_{H}\rangle}}
\]
is the kernel of $\pi_{2}$, but $\pi_{2}$ need not be its cokernel, unless $H$ is cocommutative. Hence this diagram does not represent a short exact sequence, and so neither Theorem~4.1 nor Theorem~4.2 in~\cite{Molnar} saying that $H\times H\cong H\tensor H$ applies.

We conclude that it makes no sense to study protomodular objects in $\BiAlg_{\K}$ or in $\Hopf_{\K}$, and we thus restrict our attention to the cocommutative case. 

\section*{Acknowledgements}
Thanks to Jos\'e Manuel Fern\'andez Vilaboa, Isar Goyvaerts, Marino Gran, James R.~A.~Gray, Gabriel Kadjo and Joost Vercruysse for fruitful discussions and useful comments. We would also like to thank the University of Cape Town and Stellenbosch University for their kind hospitality during our stay in South Africa.

%\bibliography{tim}
%\bibliographystyle{amsplain}

\providecommand{\noopsort}[1]{}
\providecommand{\bysame}{\leavevmode\hbox to3em{\hrulefill}\thinspace}
\providecommand{\MR}{\relax\ifhmode\unskip\space\fi MR }
\providecommand{\MRhref}[2]{%
  \href{http://www.ams.org/mathscinet-getitem?mr=#1}{#2}
}
\providecommand{\href}[2]{#2}

\end{document}